\documentclass[12pt]{amsart}

\usepackage[matrix,arrow,curve,frame]{xy}
\usepackage{amsmath,amsthm,amssymb,enumerate}
\usepackage{latexsym}
\usepackage{amscd}
\usepackage[backref]{hyperref}
\usepackage{euscript}
\usepackage{fullpage}
\usepackage{color}

\usepackage[all]{xypic}

\newtheorem{theorem}{Theorem}[section]
\newtheorem{lemma}[theorem]{Lemma}

\newtheorem{definition}[theorem]{Definition}

\theoremstyle{example}

\newtheorem{remark}[theorem]{Remark}

\numberwithin{equation}{section}

\newcommand{\beq}{\begin{equation}}
\newcommand{\eeq}{\end{equation}}

\newcommand{\TT}{\mathbb{T}}
\newcommand{\ZZ}{\mathbb{Z}}

\DeclareMathOperator{\Tr}{Tr}

\renewcommand{\cL}{\mathcal{L}}
\newcommand{\cG}{\mathcal{G}}

\newcommand{\gH}{\mathfrak{H}}
\newcommand{\gI}{\mathfrak{I}}

\newcommand{\ometil}{\widetilde{\Omega}}

\newcommand{\bu}{\bullet}

\newcommand {\be}{\begin{equation}}
\newcommand {\ee}{\end{equation}}
\newcommand{\h}{\begin{eqnarray*}}
\newcommand{\e}{\end{eqnarray*}}

\begin{document}


\title
{Exotic Twisted Equivariant K-Theory}


\author{Fei Han}
\address{Department of Mathematics,
National University of Singapore, Singapore 119076}
\email{mathanf@nus.edu.sg}

 \author{Varghese Mathai}
\address{School of Mathematical Sciences,
University of Adelaide, Adelaide 5005, Australia}
\email{mathai.varghese@adelaide.edu.au}

\date{}

\maketitle

\begin{abstract}
In this paper we introduce  {\em exotic twisted $\TT$-equivariant K-theory} of loop space $LZ$ depending on  the (typically non-flat) holonomy line bundle $\cL^B$ on $LZ$ of a gerbe with connection on $Z$. 
We define an exotic twisted $\TT$-equivariant Chern character on the exotic twisted $\TT$-equivariant K-theory of $LZ$ that maps to the exotic twisted $\TT$-equivariant cohomology of $LZ$ as previously defined in \cite{HM15}.
\end{abstract}

\tableofcontents

\section*{Introduction}

In \cite{HM15}, we introduced 
{\em exotic twisted $\TT$-equivariant cohomology} for the loop space $LZ$ of a smooth manifold $Z$
via the invariant differential forms on $LZ$ with coefficients in the (typically non-flat) holonomy line bundle $\cL^B$ of a gerbe with connection, with differential an equivariantly flat superconnection $ \nabla^{\cL^B} -\iota_K + \bar H $ in the sense of \cite{MQ, Q}, where $K$ is the rotation vector field on $LZ$ and $\bar H$ is a degree 3 circle-invariant form on $LZ$ that is completely determined by $H$, the curvature of the gerbe, cf. \cite{HM15}.

This exotic twisted $\TT$-equivariant cohomology theory has two applications. 

Firstly, we introduced in \cite{HM15} the twisted Bismut-Chern character form, generalising \cite{B85}, which is a loop space refinement of the twisted Chern character form in \cite{BCMMS} and represents classes in the completed periodic {exotic twisted $\TT$-equivariant cohomology} $h_\TT^\bullet(LZ, \nabla^{\cL^B}:\bar H)$ of $LZ$. See also 
\cite{FH08,TWZ,S,LMRT, GJP} for other interesting interpretations and extensions of the Bismut-Chern character. 
More precisely, we define these in such a way that the following diagram commutes,
\beq \label{twistedBC_H}
\xymatrix{K^\bullet(Z,H)\ar[dr]_{Ch_H}\ar[rr]^{BCh_H}&&
h_\TT^\bullet(LZ, \nabla^{\cL^B}:\bar H)\ar[dl]^{res}\\
&H^\bu(\Omega(Z)[[u, u^{-1}]], d+u^{-1}H)&}
\eeq
where $res$ is the localisation map, $\mathrm{degree}(u)=2$.

Secondly, in \cite{HM15}  we established a localisation theorem (about the map $res$) for the completed periodic exotic twisted $\TT$-equivariant cohomology for loop spaces and apply it to establish T-duality in a background flux  in type II String Theory from a loop space perspective. Continuing along these lines, we recently used in \cite{HM17} the exotic twisted $\TT$-equivariant cohomology to enhance T-duality on twisted differential forms on circle bundles, where we also showed that T-duality exchanges winding and momentum in a background flux for the first time in the model of \cite{BEM04a, BEM04b}. For an alternate approach to T-duality on loop space using the twisted chiral de Rham cohomology instead, see \cite{LM}. 

There are several approaches in the literature to the K-theory of loop spaces, and we mention two of them here. The first is \cite{Bry90}, who considers Virasoro
equivariant (infinite dimensional) vector bundles $E$ over loopspace such that the restriction to the constant loops $E\Big|_M$ decomposes as a direct sum $\bigoplus_n E_n$ under the action of the infinitismal generator of the rotation group, where each $E_n$ is a finite rank vector bundle and $E_n=0$ for $n<n_0$ for some $n_0$. The second is related to  Chas-Sullivan string topology, cf. \cite{KLW}.

In this paper, we introduce {\em exotic twisted $\TT$-equivariant K-theory}, $K_{\TT}^0(LZ, \nabla^{\cL^B}:\mathcal{G})$, for the loop space $LZ$, where $\mathcal{G}$ is the weak $\TT$-invariant gerbe on $LZ$ whose Dixmier-Douady class is $\bar H$. We also 
define the  {\em exotic twisted $\TT$-equivariant Chern character},
$$
Ch_{\nabla^{\cL^B}:\mathcal{G}}: K_{\TT}^0(LZ, \nabla^{\cL^B}:\mathcal{G}) \longrightarrow h_\TT^{even}(LZ, \nabla^{\cL^B}:\bar H)
$$
that makes the following diagram commute along the solid arrows (see Remark \ref{commute}):
\beq\label{twistedBC_H2}
\xymatrix{& K_{\TT}^0(LZ, \nabla^{\cL^B}:\mathcal{G}) \ar[dl]_{res} \ar[dr]^{Ch_{\nabla^{\cL^B}:\mathcal{G}}} \\
K^\bullet(Z,H)\ar[dr]_{Ch_H} \ar@{-->}[rr]^{BCh_H}&&
h_\TT^\bullet(LZ, \nabla^{\cL^B}:\bar H)\ar[dl]^{res}\\
&H^\bu(\Omega(Z)[[u, u^{-1}]], d+u^{-1}H)&}
\eeq
It follows that  the exotic twisted $\TT$-equivariant K-theory is the correct version of K-theory that corresponds via a Chern character map to the
exotic twisted $\TT$-equivariant cohomology as defined in \cite{HM15}. However we would like to point out that the map $BCh_H$ 
does not make the upper triangle of Diagram (\ref{twistedBC_H2}) commutative (see Remark \ref{noncommute}).

The plan of this paper is as follows.

In Section \ref{1gerbe}, we introduce the  concept of {\em weak $\TT$-invariant gerbes} and study the coupling of them to $\TT$-equivariant line bundles on possibly infinite dimensional good $\TT$-manifolds. A pair consisting of of coupled weak $\TT$-invariant gerbe  and $\TT$-equivariant line bundle will be the initial input data for an exotic twisted $\TT$-equivariant K-theory (see Section \ref{3K-thy}).

In Section \ref{2coh}, we establish the correspondence between the exotic twisted $\TT$-equivariant cohomology, concerning differential forms on $M$ with coefficients in a complex line bundle $\xi$, and certain cohomology theory concerning differential forms on $S\xi$, the circle bundle of $\xi$ over $M$ (see Theorem \ref{fundtheorem}). Such a transition from $M$ to $S\xi$ is crucial: when we attempted to develop the exotic twisted $\TT$-equivariant K-theory, we realized that it is difficult to define it on $M$ itself, instead one needs to work on the circle bundle $S\xi$. The circle bundle is much larger than $M$ and allows us more room to construct the correct K-theory, which possesses a Chern character landing into the exotic twisted $\TT$-equivariant cohomology.

In Section \ref{3K-thy}, we introduce exotic twisted $\TT$-equivariant K-theory for possibly infinite dimensional $\TT$-manifolds,
and the exotic twisted $\TT$-equivariant Chern character that lands into exotic twisted $\TT$-equivariant cohomology. We also establish the transgression formulae in this context, using a new version of Chern-Simons forms. The odd degree analogue of the theory is also established in this section. 

\bigskip

\noindent{\em Acknowledgements.}
The first author was partially supported
by the grant AcRF R-146-000-218-112 from National University
of Singapore. He would also like to thank Professor Weiping Zhang and Dr. Qin Li for helpful discussion. 
The second author was partially supported by funding from the Australian Research Council, through the 
Australian Laureate Fellowship FL170100020. The authors are grateful to the referee for very useful feedback.


\section{Coupling of $\TT$-equivariant line bundles and weak $\TT$-invariant gerbes} \label{1gerbe}
Let $M$ be a (possibly infinite dimensional) $\TT$-manifold. We call $M$ a {\bf good $\TT$-manifold} if $M$ has an open cover $\{U_\alpha\}$ such that all finite intersections $U_{\alpha_0\alpha_1\cdots\alpha_p}=U_{\alpha_0}\cap U_{\alpha_1}\cdots U_{\alpha_p}$ have trivial $\TT$-equivariant homotopy groups, for $j=0$ and $j\ge 2$. Let $K$ be the Killing vector field of the $\TT$-action. Denote by $L_K, \iota_K$ the Lie derivative and  contraction along the direction $K$ respectively.

\begin{definition} \label{gerbe}
The system $(\{U_\alpha\}, H, B_\alpha, A_{\alpha\beta})$ is called a gerbe on $M$, if
$$ 
H \in \Omega^3(M), \ B_\alpha\in \Omega^2(U_\alpha), \ \ A_{\alpha\beta}\in \Omega^1(U_{\alpha\beta}),
$$
such that $\frac{1}{2\pi i}H$ has integral period,
\be  
\begin{split}
&H=dB_\alpha\ \ \mathrm{on}\ U_\alpha,\\
&B_\alpha-B_\beta=dA_{\alpha\beta}\ \ \mathrm{on}\ U_{\alpha\beta},\\
 \end{split}
\ee
and  there exist $C_{\alpha\beta\gamma}\in C^{\infty}(U_{\alpha\beta\gamma}, U(1))$ such that 
$$A_{\alpha\beta}+A_{\beta\gamma}-A_{\alpha\gamma}=d\ln C_{\alpha\beta\gamma}.$$
\end{definition}
It is easy to see that different choices of $C_{\alpha\beta\gamma}$ differ by a $U(1)$-valued constant scalar on each connected component of $U_{\alpha\beta\gamma}$.

\begin{remark} \label{general} Our definition of a gerbe here is slightly more general than the gerbe in the usual sense. We don't require 
$C_{\beta\gamma\delta}C^{-1}_{\alpha\gamma\delta}C_{\alpha\beta\delta}C^{-1}_{\alpha\beta\gamma}=1$ on each nonempty intersection $U_\alpha\cap U_{\beta}\cap U_\gamma \cap U_\delta$. 
\end{remark}

\begin{definition} A gerbe $(\{U_\alpha\}, H, B_\alpha, A_{\alpha\beta})$ is called a {\bf  weak  $\TT$-invariant gerbe} on $M$ if \newline
(i) $H, B_\alpha, A_{\alpha\beta}$ are all $\TT$-invariant;\newline
(ii) $\iota_KA_{\alpha\beta}+\iota_KA_{\alpha\beta}-\iota_KA_{\alpha\gamma}$ takes values in $2\pi i\cdot\ZZ$ on each connected component of $U_{\alpha\beta\gamma}$.
\end{definition}

\begin{remark} The second condition is equivalent to 
$$L_KC_{\alpha\beta\gamma}=2\pi i nC_{\alpha\beta\gamma}$$ for some $n\in \ZZ$ on each connected component of $U_{\alpha\beta\gamma}$. Actually we have 
$$\iota_KA_{\alpha\beta}+\iota_KA_{\alpha\beta}-\iota_KA_{\alpha\gamma}=\iota_K\left( C^{-1}_{\alpha\beta\gamma}dC_{\alpha\beta\gamma}\right)=C^{-1}_{\alpha\beta\gamma} \iota_K dC_{\alpha\beta\gamma}=C^{-1}_{\alpha\beta\gamma}L_KC_{\alpha\beta\gamma}.$$
If all the $n$ is equal to $0$, i.e. $C_{\alpha\beta\gamma}$'s are $\TT$-invariant, we call it a {\bf $\TT$-invariant gerbe.}
\end{remark}

Let $\xi$ be a $\TT$-equivariant complex line bundle over $M$ equipped with a $\TT$-invariant connection $\nabla^\xi$. 

\begin{definition} \label{coupled}
The $\TT$-equivariant line bundle $(\xi, \nabla^{\xi})$ and the weak  $\TT$-invariant gerbe \\$(\{U_\alpha\}, H, B_\alpha, A_{\alpha\beta})$ are  said to be {\bf coupled on $M$} if under some local basis $\{s_\alpha\}$ of $\xi|_{U_\alpha}$, \newline
(i) $-\iota_K B_\alpha$ is the connection 1-form of $\nabla^{\xi}$ on $U_\alpha$ for each $\alpha$;\newline
(ii) $e^{-\iota_K A_{\alpha\beta}}$ is the transition function of $\xi$ on $U_{\alpha\beta}$ for each $\alpha, \beta$.
\end{definition}

\begin{lemma} If the $\TT$-equivariant line bundle $(\xi, \nabla^{\xi})$ and the weak $\TT$-invariant gerbe \\$(\{U_\alpha\}, H, B_\alpha, A_{\alpha\beta})$ are coupled on $M$, then the equivariant super connection $\nabla^{\xi}-u\iota_K+u^{-1}H$ on $\xi$ is equivariantly flat, i.e.
\be (\nabla^{\xi}-u\iota_K+u^{-1}H)^2+uL_K^\xi=0. \ee
\end{lemma}
\begin{proof} The proof is similar to the proof of Lemma 1 in \cite{HM15}.

\end{proof}

We provide some examples of coupled $\TT$-equivariant line bundles and weak $\TT$-invariant gerbes.

$\, $

\noindent {\bf Example 1.}  Let $Z$ be a smooth manifold. Let $\{U_\alpha\}$ be a {\em Brylinski open cover} of $Z$, i.e. $\{U_\alpha\}$ is a maximal open cover of $Z$ with the property that $H^i(U_{\alpha_I})=0$ for $i=2, 3$ where $U_{\alpha_I} = \bigcap_{i\in I} U_{\alpha_i}, \,$ $|I|<\infty.$
Then the free loop space $LZ$ is good $\TT$-manifold with the open cover $\{LU_\alpha\}$, where the $\TT$-action is the loop rotating action. 

Let 
\beq\label{eqn:trans}
\tau: \Omega^\bullet(U_{\alpha_I} ) \longrightarrow \Omega^{\bullet-1}(LU_{\alpha_I} )
\eeq
be the transgression map 
\beq
\tau(\xi_I) = \int_{\TT} ev^*(\xi_I), \qquad \xi_I \in \Omega^\bullet(U_{\alpha_I} ).
\eeq
Here $ev$ is the evaluation map
\beq
ev: \TT \times LZ \to Z: (t, \gamma)\to \gamma(t).
\eeq

Let $\omega \in \Omega^i(Z)$. Define $\hat\omega_s \in \Omega^i(LZ)$ for $s\in [0,1]$ by
\beq
\hat\omega_s(X_1, \ldots, X_i)(\gamma) = \omega(X_1\big|_{\gamma(s)}, \ldots, X_i\big|_{\gamma(s)})
\eeq
for $\gamma\in LZ$ and $X_1, \ldots, X_i$ vector fields on $LZ$ defined near $\gamma$. Then one checks that
$d \hat\omega_s = {\widehat{d\omega}}_s$. The $i$-form, averaging $\omega$ on the loop space, 
\beq
\bar\omega=\int_{0}^1  \hat\omega_s ds \in \Omega^i(LZ)
\eeq
is $\TT$-invariant, that is, $L_K\left(\bar\omega\right) = 0$. Moreover $\tau(\omega) = \iota_K\bar\omega$. We call $\bar \omega$ the average of $\omega$.

Let $(\{U_\alpha\}, H, B_\alpha, A_{\alpha\beta})$ be a gerbe on $Z$. Associated to this gerbe, there exists a pair of coupled $\TT$-equivariant line bundle and weak $\TT$-invariant gerbe on $LZ$.

The holonomy of this gerbe is a $\TT$-equivariant line bundle $\cL^B \to LZ$ over the loop space $LZ$, whose construction is 
detailed in Section 6.2.1 in \cite{Bry}. $\cL^B$ has  $\TT$-invariant  local sections $\{\sigma_\alpha\}$ with respect to $\{LU_\alpha\}$ such that the transition functions are $\{e^{-\int_0^1\iota_K A_{\alpha_\beta}} = e^{-\tau(A_{\alpha\beta})}\}$, i.e. $\sigma_\alpha=e^{-\int_0^1\iota_K A_{\alpha_\beta}}\sigma_\beta$.  $\cL^B$ comes with a natural connection, whose definition with respect to the open cover  $\{LU_\alpha\}$  is
\beq
\nabla^{\cL^B} = d-\iota_K \bar B_\alpha=d-\tau(B_\alpha).
\eeq
For more details, cf. 6.2 in \cite{Bry}.

On the other hand, averaging the gerbe $(\{U_\alpha\}, H, B_\alpha, A_{\alpha\beta})$ gives rise to a gerbe  $$(\{LU_\alpha\}, \bar{H}, \bar{B}_\alpha, \bar{A}_{\alpha\beta})$$ on $LZ$. First it is not hard to see that $\frac{1}{2\pi i}\bar{H}$ still has integral periods. It is evident that 
\be  
\begin{split}
&\bar{H}=d\bar{B}_\alpha\ \ \mathrm{on}\ LU_\alpha,\\
&\bar{B}_\alpha-\bar{B}_\beta=d\bar{A}_{\alpha\beta}\ \ \mathrm{on}\ LU_{\alpha\beta}.\\
 \end{split}
\ee
If on $U_{\alpha\beta\gamma}$,
\be \label{third rel} A_{\alpha\beta}+A_{\beta\gamma}-A_{\alpha\gamma}=d\ln C_{\alpha\beta\gamma}, \ee
then 
\be \label{integral} \iota_K\bar{A}_{\alpha\beta}+\iota_K\bar{A}_{\beta\gamma}-\iota_K\bar{A}_{\alpha\gamma}=\tau(d\ln C_{\alpha\beta\gamma})\in2\pi i\ZZ \ee
on each connected component of $LU_{\alpha\beta\gamma}$. By (\ref{integral}), if $x_0$ is a fixed loop in $U_{\alpha\beta\gamma}$ and $x$ is any loop in $U_{\alpha\beta\gamma}$, then 
\be e^{\int_{x_0}^{x} (\bar{A}_{\alpha\beta}+\bar{A}_{\beta\gamma}-\bar{A}_{\alpha\gamma})}\ee
does not depend on the choice of paths from $x_0$ to $x$ in $LU_{\alpha\beta\gamma}$. By (\ref{third rel}), it is not hard to see that $\int_{x_0}^{x} (\bar{A}_{\alpha\beta}+\bar{A}_{\beta\gamma}-\bar{A}_{\alpha\gamma})$ is pure imaginary.
Then we further have
\be \bar{A}_{\alpha\beta}+\bar{A}_{\beta\gamma}-\bar{A}_{\alpha\gamma}=d\ln e^{\int_{x_0}^{x} (\bar{A}_{\alpha\beta}+\bar{A}_{\beta\gamma}-\bar{A}_{\alpha\gamma})},\ee
where $e^{\int_{x_0}^{x} (\bar{A}_{\alpha\beta}+\bar{A}_{\beta\gamma}-\bar{A}_{\alpha\gamma})}$ is an $U(1)$-valued function on $LU_{\alpha\beta\gamma}$. Therefore $(\{LU_\alpha\}, \bar{H}, \bar{B}_\alpha, \bar{A}_{\alpha\beta})$ is a gerbe on $LZ$.

It is obvious that $\bar{H}, \bar{B}_\alpha, \bar{A}_{\alpha\beta}$ are all $\TT$-invariant. Combining (\ref{integral}), we see that the gerbe $(\{LU_\alpha\}, \bar{H}, \bar{B}_\alpha, \bar{A}_{\alpha\beta})$ is a weak $\TT$-invariant gerbe on $LZ$.

As with the  $\TT$-invariant local sections, the local connection 1-form of $\left(\cL^B, \nabla^{\cL^B}\right)$ is 
$$-\tau(B_\alpha)=-\iota_K \bar{B}_\alpha,$$ 
and the transition function of $\cL^B$ is 
$$e^{-\int_0^1\iota_K A_{\alpha_\beta}}=e^{-\iota_K\bar{A}_{\alpha\beta} },$$
we see that $\left(\cL^B, \nabla^{\cL^B}\right)$ and $(\{LU_\alpha\}, \bar{H}, \bar{B}_\alpha, \bar{A}_{\alpha\beta})$ are coupled on $LZ$.

 $\, $

\noindent {\bf Example 2.}  In \cite{BEM04a,BEM04b}, T-duality in a background flux has the following settings. There is a principal circle bundle $\TT \to Z \stackrel{\pi}{\to} X$ with a $\TT$-invariant connection $\Theta$ and a background $\TT$-invariant flux $H$, which is a $\TT$-invariant closed 3-form on $Z$.  Let $\{U_\alpha\}$ be a good cover of $X$. The cover $\{\pi^{-1}(U_\alpha)\}$ makes $Z$ a good $\TT$-manifold. 

The T-dual space $\hat\TT\to\hat Z \stackrel{\hat\pi}{\to} X$ is a principal circle bundle with a $\hat \TT$-invariant connection $\hat\Theta$ and a background $\hat\TT$-invariant flux $\hat H$. The cover $\{\hat\pi^{-1}(U_\alpha)\}$ makes $\hat Z$ a good $\TT$-manifold. 

Denote $v, \hat{v}$ the Killing vector field on $Z, \hat{Z}$ respectively.  The gerbe $(\{\pi^{-1}(U_\alpha)\}, H, B_\alpha,  A_{\alpha\beta})$ on $Z$ and  the gerbe $(\{\hat\pi^{-1}(U_\alpha)\}, \hat H, \hat B_\alpha, \hat A_{\alpha\beta})$  on $\hat Z$ satisfy the following relations
\be e^{-\iota_v A_{\alpha\beta}}=\hat g_{\alpha\beta}, \ -\iota_v {B_\alpha}=\hat \eta_\alpha, \ \iota_v H=F^{\hat\Theta}\ee  and
\be  e^{-\iota_v \hat A_{\alpha\beta}}=g_{\alpha\beta}, \  -\iota_{\hat v} {\hat B_\alpha}=\eta_\alpha, \ \iota_{\hat v} {\hat H}=F^{\Theta}, \ee
where  $\hat g_{\alpha\beta}$ is the transition functions of the bundle $\hat{Z}$, $\hat \eta_\alpha$ is the local connection 1-form of $\hat \Theta$ on $U_\alpha$, $F^{\hat\Theta}$ is the curvature 2-form of $\hat\Theta$ on $X$ and the similar meaning for the notations without hats on the dual side.  

In the setting, $B_\alpha, A_{\alpha\beta}$ are all chosen to be $\TT$-invariant. Moreover as 
$e^{-\iota_v A_{\alpha\beta}}=\hat g_{\alpha\beta}$, we conclude that $\iota_vA_{\alpha\beta}+\iota_vA_{\alpha\beta}-\iota_vA_{\alpha\gamma}$ takes values in $2\pi i\cdot\ZZ$ on each  $U_{\alpha\beta\gamma}$. Therefore $(\{\pi^{-1}(U_\alpha)\}, H, B_\alpha,  A_{\alpha\beta})$ is a weak $\TT$-invariant gerbe on $Z$. Similarly $(\{\hat\pi^{-1}(U_\alpha)\}, \hat H, \hat B_\alpha, \hat A_{\alpha\beta})$ is a weak $\hat\TT$-invariant gerbe on $\hat Z$.

$(\hat Z, \hat\Theta)$ and the standard representation of the circle on complex plane give rise to a complex line bundle with connection $(\hat\xi, \nabla^{\hat \xi})$ on $X$. Dually, there is a similar $(\xi, \nabla^{\xi})$ on $X$ coming from $(Z, \Theta)$. As  
$$e^{-\iota_v A_{\alpha\beta}}=\hat g_{\alpha\beta}, -\iota_v {B_\alpha}=\hat \eta_\alpha,$$
the $\TT$-equivariant line bundle $(\pi^*\hat{\xi}, \pi^*\nabla^{\hat{\xi}})$ and the weal $\TT$-invariant gerbe $(\{\pi^{-1}(U_\alpha)\}, H, B_\alpha, A_{\alpha\beta})$ are coupled on $Z$. Dually, the $\hat\TT$-equivariant line bundle $(\hat{\pi}^*\xi, \hat{\pi}^*\nabla^{\xi})$ and the $\hat\TT$-invariant gerbe $(\{\hat\pi^{-1}(U_\alpha)\}, \hat H, \hat B_\alpha, \hat A_{\alpha\beta})$ are coupled on $\hat{Z}$.

 $\, $


\section{Exotic twisted equivariant cohomology and $U(1)$-bundles}\label{2coh}

Let $M$ be a good $\TT$-manifold. Let $\xi\to M$ be a $\TT$-equivariant Hermitian line bundle over $M$ equipped with a $\TT$-invariant Hermitian connection $\nabla^\xi$. Let $H\in \Omega^3_{cl}(M)$ be a $\TT$-invariant closed 3-form such that the equivariant super connection $\nabla^{\xi}-u\iota_K+u^{-1}H$ is equivariantly flat, i.e.
 \be (\nabla^{\xi}-u\iota_K+u^{-1}H)^2+uL_K^\xi=0,\ee
 where $u$ is a degree 2 indeterminate. 
 
In the previous section, we have seen examples that satisfy these settings. 

Let $\pi: S\xi\to M$ be the principal $U(1)$-bundle of $\xi$. Let $v$ be the vertical tangent vector field on $S\xi$, i.e. the Killing vector field of the $U(1)$-action.

It is clear that $S\xi$ also admits the induced $\TT$-action. As the action of $\TT$ on the fibers of $\xi$ is linear, i.e. $g(\lambda\cdot v)=\lambda \cdot g(v), \forall g\in \TT, \lambda\in U(1)$, one deduces that the $\TT$-action and the $U(1)$-action commute. Therefore we have 
\be [K, v]=0.\ee

The condition $(\nabla^{\xi}-u\iota_K+u^{-1}H)^2+uL_K^\xi=0$ is equivalent to the following three equalities,
\be \label{fundrel}
\left\{
\begin{split}
&\mu^\xi_K=L_K^\xi-[\nabla^\xi, \iota_K]=L^\xi_K-\nabla^\xi_K=0\\
&(\nabla^\xi)^2-\iota_K H=0\\
&dH=0
\end{split}
\right.
\ee

Let $\Theta$ be the connection 1-form on $S\xi$ for $(\xi, \nabla^\xi)$. 
\begin{lemma}
\be \iota_K\Theta=0, \ L_K \Theta=0\ee and
\be d\Theta=\iota_K \pi^*H.\ee
\end{lemma}
\begin{proof} Let $\{U_\alpha\}$  be a $\TT$-cover of $M$. Choose a $\TT$-invariant local basis $s_\alpha$ of $\xi$ on $U_\alpha$. Let $\eta_\alpha$ be the connection 1-form corresponding to $s_\alpha$. By the first relation in (\ref{fundrel}), we have
$$ 0=\mu^\xi_K(s_\alpha)=(L_K^\xi-[\nabla^\xi, \iota_K])(s_\alpha)=(\iota_K \eta_\alpha)\otimes s_\alpha,$$
and therefore we have
\be \iota_K\eta_\alpha=0. \ee
As $s_\alpha$ is $\TT$-invariant, we get a local $\TT$-equivariant diffeomorphism $\phi_\alpha: U_\alpha\times S^1\to \pi^{-1}(U_\alpha)$ such that on the left hand side, $\TT$ only acts on $U_\alpha$. Then as $\phi_\alpha^*(\Theta)|_{U_\alpha\times S^1}=\eta_\alpha+d\theta$, we deduce that
$$ \iota_K\Theta=0, \ L_K \Theta=0.$$
By the second relation in (\ref{fundrel}), we get
$$ d\Theta+\frac{1}{2}\Theta^2-\iota_K \pi^*H=0$$
or
$$ d\Theta=\iota_K \pi^*H.$$
\end{proof}

Consider the $C^\infty(M)$-module
\be\ometil^*(S\xi):=\{\omega\in \Omega^*(S\xi)|\iota_v\omega=0, L_v\omega=-\omega\}.  \ee
\begin{theorem}
$$\left(\ometil^*(S\xi)^{\TT}[[u, u^{-1}]], d-\iota_V-u\iota_K+\Theta+u^{-1}\pi^*H\right) $$
is a chain complex.

\end{theorem}
\begin{proof}We need to show that:\newline
(i) if $\omega\in \ometil^*(S\xi)^{\TT}$, then 
$$(d-\iota_v-u\iota_K+\Theta+u^{-1}\pi^*H)\omega\in \ometil^*(S\xi)^{\TT};$$
(ii)$$(d-\iota_v-u\iota_K+\Theta+u^{-1}\pi^*H)^2+uL_K=0. $$

(i) holds as we have following three equalities,
\be
\begin{split}
&[d-\iota_v-u\iota_K+\Theta+u^{-1}\pi^*H, \iota_v]\\
=&L_v-[\iota_v, \iota_v]-u\iota_{[K,v]}+\iota_v\Theta+u^{-1}\iota_v(\pi^*H)\\
=&L_v+\iota_v\Theta\\
=&0\ \mathrm{on}\ \ometil^*(S\xi);
\end{split}
\ee

\be
\begin{split}
&[d-\iota_v-u\iota_K+\Theta+u^{-1}\pi^*H, L_v]\\
=&[d, L_v]-\iota_{[v, v]}-u\iota_{[K,v]}+L_v\Theta+u^{-1}L_v(\pi^*H)\\
=&0;
\end{split}
\ee
and 

\be
\begin{split}
&[d-\iota_v-u\iota_K+\Theta+u^{-1}\pi^*H, L_K]\\
=&[d, L_K]-\iota_{[v, K]}-u\iota_{[K,K]}+L_K\Theta+u^{-1}L_K(\pi^*H)\\
=&0.
\end{split}
\ee

To show (ii), we have 
\be
\begin{split}
&(d-\iota_v-u\iota_K+\Theta+u^{-1}\pi^*H)^2\\
=&(d-\iota_v-u\iota_K)^2+(d-\iota_v-u\iota_K)(\Theta+u^{-1}\pi^*H)+(\Theta+u^{-1}\pi^*H)^2\\
=&-L_v-uL_K+d\Theta-\iota_v\Theta-\pi^*\iota_KH\\
=&(-L_v-\iota_v\Theta)+(d\Theta-\iota_K\pi^*H)-uL_K\\
=&-uL_K \ \mathrm{on}\ \ometil^*(S\xi).
\end{split}
\ee

\end{proof}
 
Let $\pi^*\xi$ be the pull back bundle of $\xi$ on $S\xi$. Clearly this is a trivial bundle which has a canonical global nowhere vanishing section
$$ \gamma: (x, y)\to y,  \ \ x\in M, y\in \pi^{-1}(x). $$

Consider the map
\be \label{def-f} f: \Omega^*(M, \xi)\to \Omega^*(S\xi), \ \ \omega\mapsto \gamma^{-1}\cdot \pi^*\omega.\ee

Let $\{U_\alpha\}$ be an $\TT$-cover of $M$. Let $s_\alpha$ be a $\TT$-invariant local basis of the $\xi$ on $U_\alpha$. Suppose $\omega|_{U\alpha}=\omega_\alpha\otimes s_\alpha.$ Then on $\pi^{-1}U_\alpha\cong U_\alpha\times S^1$, 
$$ \pi^*\omega=\pi^*(\omega_\alpha)\otimes \pi^*(s_\alpha), \ \ \gamma=z\cdot \pi^*(s_\alpha), \ \ v=z\, \partial_z,$$
where $z$ is the complex coordinate on $S^1$.  
Therefore on $\pi^{-1}U_\alpha\cong U_\alpha\times S^1$, 
$$\gamma^{-1} \cdot \pi^*\omega=z^{-1}\pi^*(\omega_\alpha)$$
and 
$$\iota_v(\gamma^{-1} \cdot \pi^*\omega)=0, \ \ L_v(\gamma^{-1} \cdot \pi^*\omega)=-\gamma^{-1} \cdot \pi^*\omega. $$

Hence we see that
$$\mathrm{Im}(f)= \ometil^*(S\xi), \ \mathrm{ker}(f)=\{0\}$$
and therefore get an isomorphism of $C^\infty(M)$-modules:
\be f: \Omega^*(M, \xi)\rightarrow  \ometil^*(S\xi). \ee

Since $\gamma$ is a $\TT$-invariant global section of $\pi^*\xi$, we see that $f$ sends $\TT$-invariant invariant parts to $\TT$-invariant invariant parts. Hence we get an isomorphism of $C^\infty(M)$-modules, which we still denote by $f$:
\be \label{fundmap} f: \Omega^*(M, \xi)^{\TT}\rightarrow  \ometil^*(S\xi)^{\TT} .\ee

\begin{theorem} \label{fundtheorem}
\be f:\left(\Omega^*(M, \xi)^{\TT}[[u, u^{-1}]], \nabla^\xi-u\iota_K+u^{-1}H  \right)\rightarrow \left(\ometil^*(S\xi)^{\TT}[[u, u^{-1}]], d-\iota_v-u\iota_K+\Theta+u^{-1}\pi^*H\right)  \ee
is a chain map and induces an isomorphism on cohomology
\be f^*:\  h_{\TT}^{*}(M, \nabla^\xi:H)\rightarrow H^*\left(\ometil^*(S\xi)^{\TT}[[u, u^{-1}]], d-\iota_v-u\iota_K+\Theta+u^{-1}\pi^*H\right),\ee
where $h_{\TT}^{*}(M, \nabla^\xi:H)$ is the {\bf completed periodic exotic twisted $\TT$-equivariant cohomology} \cite{HM15}.
\end{theorem}
\begin{proof} Let $\omega\in \Omega^*(M, \xi)^{\TT}[[u, u^{-1}]]$. We have 
\be
\begin{split}
&(d-\iota_v-u\iota_K+\Theta+u^{-1}\pi^*H)(f(\omega))\\
=&(d-\iota_v-u\iota_K+\Theta+u^{-1}\pi^*H)(\gamma^{-1}\cdot \pi^*\omega)\\
=&(d-u\iota_K+\Theta)(\gamma^{-1}\cdot \pi^*\omega)+u^{-1}\pi^*H (\gamma^{-1}\cdot \pi^*\omega).
\end{split}
\ee

Let $\{U_\alpha\}$ be an $\TT$-cover of $M$. Let $s_\alpha$ be a $\TT$-invariant local basis of the $\xi$ on $U_\alpha$. Suppose $\omega|_{U\alpha}=\omega_\alpha\otimes s_\alpha.$

Then 
\be   
\begin{split}
&d(\gamma^{-1}\cdot \pi^*\omega)\\
=&d(\pi^*\omega_\alpha\cdot (\gamma^{-1}\cdot \pi^*s_\alpha))\\
=&\pi^*(d\omega_\alpha)(\gamma^{-1}\cdot \pi^*s_\alpha)-\pi^*(\omega_\alpha)d(\gamma^{-1}\cdot \pi^*s_\alpha).
\end{split}
\ee
Therefore locally, we have
\be   
\begin{split}
&d(\gamma^{-1}\cdot \pi^*\omega)+\Theta(\gamma^{-1}\cdot \pi^*\omega)\\
=&\pi^*(d\omega_\alpha)(\gamma^{-1}\cdot \pi^*s_\alpha)-\pi^*(\omega_\alpha)d(\gamma^{-1}\cdot \pi^*s_\alpha)+\Theta(\pi^*\omega_\alpha)(\gamma^{-1}\cdot \pi^*s_\alpha)\\
=&\pi^*(d\omega_\alpha)(\gamma^{-1}\cdot \pi^*s_\alpha)-\pi^*(\omega_\alpha)(\gamma^{-1}\cdot \pi^*\omega)[\Theta-(\gamma^{-1}\cdot \pi^*\omega)^{-1}d(\gamma^{-1}\cdot \pi^*s_\alpha)]\\
=&[\pi^*(d\omega_\alpha)-\pi^*(\omega_\alpha)\eta_\alpha](\gamma^{-1}\cdot \pi^*s_\alpha),
\end{split}
\ee
where $\eta_\alpha=\Theta-(\gamma^{-1}\cdot \pi^*\omega)^{-1}d(\gamma^{-1}\cdot \pi^*s_\alpha)$ is connection one form for the basis $s_\alpha$ of the connection $\nabla^\xi$ on $U_\alpha.$

Moreover, we have
\be   
\begin{split}
&\iota_K(\gamma^{-1}\cdot \pi^*\omega)\\
=&\iota_K(\pi^*(\omega_\alpha)(\gamma^{-1}\cdot \pi^*s_\alpha))\\
=&\iota_K(\pi^*(\omega_\alpha))(\gamma^{-1}\cdot \pi^*s_\alpha).
\end{split}
\ee

Therefore,
\be   
\begin{split}
&[d(\gamma^{-1}\cdot \pi^*\omega)+\Theta(\gamma^{-1}\cdot \pi^*\omega)+\iota_K(\gamma^{-1}\cdot \pi^*\omega)]|_{U_\alpha}\\
=&\pi^*(d\omega_\alpha+\omega_\alpha\eta_\alpha-u\iota_K\omega_\alpha)(\gamma^{-1}\cdot \pi^*s_\alpha)\\
=&\gamma^{-1}\pi^*[(d\omega_\alpha+\omega_\alpha\eta_\alpha-u\iota_K\omega_\alpha)\otimes s_\alpha ]\\
=&\gamma^{-1}\pi^*[(\nabla^\xi-u\iota_K)\omega]|_{U_\alpha}\\
=&f((\nabla^\xi-u\iota_K)\omega)|_{U_\alpha}.
\end{split}
\ee
And so we have
\be \label{comm}(d-\iota_v-u\iota_K+\Theta+u^{-1}\pi^*H)(f(\omega))=f((\nabla^\xi-u\iota_K+u^{-1}H)\omega).\ee
\end{proof}


\section{Exotic twisted equivariant $K$-theory and the Chern character}\label{3K-thy}

\subsection{Gerbe modules and twisted $K$-theories}
A geometric realization of the gerbe $\cG=(\{U_\alpha\}, H, B_\alpha, A_{\alpha\beta})$ on $M$ is $\{(L_{\alpha\beta}, \nabla^L_{\alpha\beta})\}$, a collection of trivial line bundles $L_{\alpha\beta}\to U_{\alpha\beta}$ with connections $\nabla^L_{\alpha\beta}=d+A_{\alpha\beta}$ such that on $U_{\alpha\beta\gamma}$ there are connection preserving isomorphisms
\be \label{gerbeproperty} L_{\alpha\beta} \otimes L_{\beta\gamma} \cong L_{\alpha\gamma}.\ee  
Note that as here we are using slightly more general version of gerbe (see Definition \ref{gerbe} and Remark \ref{general}), the isomorphisms $L_{\alpha\beta} \otimes L_{\beta\gamma} \cong L_{\alpha\gamma}$ are not uniquely fixed, but may differ by a multiplication by a locally constant $U(1)$-valued scalar.  Then we have
\beq\label{gerbeconn}
(\nabla^L_{\alpha\beta})^2 = F^L_{\alpha\beta} = B_\beta- B_\alpha.
\eeq

Let $E=\{E_{\alpha}\}$
be a collection of (infinite dimensional) Hilbert bundles $E_{\alpha}\to U_{\alpha}$ whose structure group is reduced to
$U_{\gI}$, which are unitary operators on the model Hilbert space $\gH$ of the form identity + trace class operator.
Here $\gI$ denotes the Lie algebra of trace class operators on $\gH$.
In addition, assume that on the overlaps $U_{\alpha\beta}$ that
there are isomorphisms
\beq
\phi_{\alpha\beta}: L_{\alpha\beta} \otimes E_\beta \cong E_\alpha,
\eeq
which are consistently defined on
triple overlaps because of the gerbe property (\ref{gerbeproperty}). More precisely, one has
\beq
(L_{\alpha\beta} \otimes L_{\beta\gamma}) \otimes E_\gamma \cong  L_{\alpha\gamma}\otimes E_\gamma    \cong E_\alpha,
\eeq
and
\beq
L_{\alpha\beta} \otimes (L_{\beta\gamma} \otimes E_\gamma) \cong L_{\alpha\beta} \otimes E_\beta    \cong E_\alpha.
\eeq

Then $\{E_{\alpha}\}$ is said to be a {\em gerbe module} for the gerbe
$\{L_{\alpha\beta}\}$. A {\em gerbe module connection} $\nabla^E$ is a collection of connections $\{\nabla^E_{\alpha}\}$ is of the form $\nabla^E_{\alpha} = d + A_\alpha^E$ where $A_\alpha^E
\in \Omega^1(U_\alpha)\otimes \gI$ whose curvature $F^E_\alpha$ on the overlaps $U_{\alpha\beta}$ satisfies
\beq
\phi_{\alpha\beta}^{-1}(F^{E_\alpha}) \phi_{\alpha\beta} =  F^{L_{\alpha\beta}} I  +    F^{E_\beta}.
\eeq
Using equation \eqref{gerbeconn}, this becomes
\beq
\phi_{\alpha\beta}^{-1}( B_\alpha I + F^E_\alpha ) \phi_{\alpha\beta} = B_{\beta} I  +    F^E_\beta.
\eeq
It follows that $\exp(-B)\Tr\left(\exp(-F^E) - I\right)$ is a globally well defined differential form on $M$
of even degree. Notice that $\Tr(I)=\infty$ which is why we need to consider the subtraction.

Let $E=\{E_{\alpha}\}$ and $E'=\{E'_{\alpha}\}$ 
be a {gerbe modules} for the gerbe $\{L_{\alpha\beta}\}$. Then an element of twisted K-theory $K^0(M, \mathcal{G})$
is represented by the pair $(E, E')$, see \cite{BCMMS}. Two such pairs $(E, E')$ and $(G, G')$ are equivalent
if $E\oplus G' \oplus K \cong E' \oplus G \oplus K$ as gerbe modules for some gerbe module $K$ for the gerbe $\{L_{\alpha\beta}\}$.
We can assume without loss of generality that these gerbe modules $E, E'$ are modeled on the same Hilbert space
$\gH$, after a choice of isomorphism if necessary.

Suppose that $\nabla^E, \nabla^{E'}$ are gerbe module connections on the gerbe modules $E, E'$ respectively. Then we can define the {\em twisted Chern character} as
\begin{align*}
Ch_H &: K^0(M, \mathcal{G}) \to H^{even}(M, H),\\
Ch_H(E, E')&= \exp(-B)\Tr\left(\exp(-F^E) - \exp(-F^{E'})\right).
\end{align*}
That this is a well defined homomorphism is explained in \cite{BCMMS, MS}. To define the twisted Chern character landing in $\left(\Omega^\bu(M)[[u, u^{-1}]]\right)_{(d+u^{-1}H)-cl}$, simply replace the above formula by
$$Ch_H(E, E')= \exp(-u^{-1}B)\Tr\left(\exp(-u^{-1}F^E) - \exp(-u^{-1}F^{E'})\right).$$

The above theory can be extended to equivariant setting with a compact group action on all the data \cite{MS}. †

\subsection{Exotic twisted equivariant $K$-theory} \label{extevenK}

Let $M$ be a good $\TT$-manifold with an $\TT$-invariant cover $\{U_\alpha\}$. Let $\xi\to M$ be a $\TT$-equivariant Hermitian line bundle over $M$ equipped with a $\TT$-invariant Hermitian connection $\nabla^\xi$. Let $\pi: S\xi\to M$ be the principal $U(1)$-bundle of $\xi$. Let $\mathcal{G}=(\{U_\alpha\}, H, B_\alpha, A_{\alpha\beta})$ be a weak $\TT$-invariant gerbe on $M$ and $\{(L_{\alpha\beta}, \nabla^L_{\alpha\beta})\}$ a geometrization of $\mathcal{G}$. Assume that $(\xi, \nabla^\xi)$ and $(\{U_\alpha\}, H, B_\alpha, A_{\alpha\beta})$ are coupled on $M$. Denote this system by $\{M, \mathcal{G}, (\xi, \nabla^\xi)\}$.

Associated to the system $\{M, \mathcal{G}, (\xi, \nabla^\xi)\}$, we will introduce a version of twisted $K$-theory and twisted Chern character in this section.

It is clear that the open cover $\{\pi^{-1}(U_\alpha)\}$ makes  $S\xi$ a good $(\TT\times U(1))$-manifold. Here to distinguish the two circle actions, we denote by $\TT$ the circle acting on the base $M$ and by $U(1)$ the circle acting on the fibers. 

Denote $\mathcal{G}^{\xi}:=(\{\pi^{-1}(U_\alpha)\}, \pi^*H, \pi^*B_\alpha, \pi^*A_{\alpha\beta})$, which is a $(\TT\times U(1))$-invariant gerbe on $S\xi$. Let $\{(\hat L_{\alpha\beta}, \nabla^{\hat L_{\alpha\beta}}=d+\pi^*A_{\alpha\beta})\}$ be the system of  $(\TT\times U(1))$-line bundles with $(\TT\times U(1))$-invariant connections on $U_{\alpha\beta}\times U(1)$, which is the geometrization of the gerbe $\mathcal{G}^\xi$.

Let $v$ be the vertical tangent vector field on $S\xi$, i.e. the Killing vector field of the $U(1)$-action. Let $K$ be the Killing vector field of the $\TT$-action. Let $u$ be a degree 2 indeterminate.

\begin{definition} \label{maindef} $E=\{E_{\alpha}, \nabla^{E_\alpha}\}$ is called
a $(\TT\times U(1))$-equivariant {\bf gerbe module with horizontal connection} for the gerbe $\{\hat L_{\alpha\beta}\}$ if \newline
(a) the $(\TT\times U(1))$-invariant connections $\nabla^{E_\alpha}$'s vanish on the vertical direction, i.e. $\nabla^{E_\alpha}_v\equiv 0$;\newline
(b) there are  $(\TT\times U(1))$-equivariant  isomorphisms
$$ 
\phi_{\alpha\beta}: \hat L_{\alpha\beta} \otimes E_\beta \cong E_\alpha,
$$
that define a gerbe module and which respect the connections. \end{definition}
Note that the isomorphisms $\{\phi_{\alpha\beta}\}$ are consistently defined on
triple overlaps because of the type  (\ref{gerbeproperty})  property of the gerbe  $\{(\hat L_{\alpha\beta}, \nabla^{\hat L_{\alpha\beta}}=d+\pi^*A_{\alpha\beta})\}$. \newline

Let $(E, E')$ and $(G, G')$ be two pairs of $(\TT\times U(1))$-equivariant gerbe modules with horizontal connections for the gerbe $\{\hat L_{\alpha\beta}\}$. We say they are equivalent, denoted by 
$$(E, E')\sim (G, G')$$ if there exists some $K$, a $(\TT\times U(1))$-equivariant gerbe modules with horizontal connection, such that 
$$E\oplus G' \oplus K \cong E' \oplus G \oplus K$$ as  $(\TT\times U(1))$-equivariant gerbe modules with horizontal connections. Clearly this is an equivalence relation. 
As usual, we define \be \hat K_{\TT}^0(M, \nabla^\xi:\mathcal{G}):=\{(E, \nabla^E, E', \nabla^{E'})\}/\{\sim\}. \ee 
If the horizontal gerbe module connections are forgotten, one defines the {\bf exotic twisted $\TT$-equivariant $K$-theory} of $\{M, \mathcal{G}, (\xi, \nabla^\xi)\}$, denoted as $K_{\TT}^0(M, \nabla^\xi:\mathcal{G})$, by
\be K_{\TT}^0(M, \nabla^\xi:\mathcal{G}):=\{(E, E')\}/\{\sim\}. \ee

Let $E=\{E_{\alpha}, \nabla^{E_\alpha}\}$ be 
a $(\TT\times U(1))$-equivariant gerbe module with horizontal connection for the gerbe $\{\hat L_{\alpha\beta}\}$. For the equivariant curvatures along the direction $v+uK$, we have
\beq
\phi_{\alpha\beta}^{-1}(F^{E_\alpha}+\mu_{v+uK}^{E_\alpha}) \phi_{\alpha\beta} =  (F^{{\hat L}_{\alpha\beta}}+\mu_{v+uK}^{{\hat L}_{\alpha\beta}}) I  +    (F^{E_\beta}+\mu_{v+uK}^{E_\beta}),\eeq
where $\mu$ stands for the moment. However
\be F^{{\hat L}_{\alpha\beta}}=\pi^*B_\beta-\pi^*B_\alpha, \ee
\be \mu_{v+uK}^{{\hat L}_{\alpha\beta}}=(\iota_v+u\iota_K)\pi^*A_{\alpha\beta}=u\iota_K  \pi^*A_{\alpha\beta}=2\pi i u\theta_\beta-2\pi i u\theta_\alpha,\ee
where $\theta_\alpha$ (resp. $\theta_\beta$) are the vertical coordinates of $\pi^{-1}(U_\alpha)$ (resp. $\pi^{-1}(U_\beta)$). So we have 
\be \label{glue} \phi_{\alpha\beta}^{-1}(F^{E_\alpha}+\mu_{v+uK}^{E_\alpha}+\pi^*B_\alpha+2\pi i u\theta_\alpha) \phi_{\alpha\beta}=F^{E_\beta}+\mu_{v+uK}^{E_\beta}+\pi^*B_\beta+2\pi i u\theta_\beta.  \ee
Therefore the forms 
$$\exp(-u^{-1}\pi^*B_\alpha-2\pi i \theta_\alpha)\Tr\left(\exp(-u^{-1}(F^{E_\alpha}+\mu_{v+uK}^{E_\alpha})) -I \right)$$ 
can be glued together as a global 
differential form in $\Omega^*(S\xi)[[u, u^{-1}]]$. 
Now let $E'=\{E'_{\alpha}\}$ 
be another $(\TT\times U(1))$-equivariant gerbe module for the gerbe $\{\hat L_{\alpha\beta}\}$. Similar to $E=\{E_{\alpha}, \nabla^{E_\alpha}\}$, the forms 
$$ \exp(-u^{-1}\pi^*B_\alpha-2\pi i \theta_\alpha)\Tr\left(\exp(-u^{-1}(F^{E'_\alpha}+\mu_{v+uK}^{E'_\alpha}))-I\right)$$ can be glued together as a global 
differential form in $\Omega^*(S\xi)[[u, u^{-1}]]$.
Then we see that
\be \label{glue1} \exp(-u^{-1}\pi^*B_\alpha-2\pi i \theta_\alpha)\Tr\left(\exp(-u^{-1}(F^{E_\alpha}+\mu_{v+uK}^{E_\alpha})) -\exp(-u^{-1}(F^{E'_\alpha}+\mu_{v+uK}^{E'_\alpha})) \right)\ee glues to a global 
differential form in $\Omega^*(S\xi)[[u, u^{-1}]]$. Simply denote this form by 
\be ch_{\nabla^\xi:\mathcal{G}}(\nabla^E, \nabla^{E'})=\exp(-u^{-1}\pi^*B-2\pi i \theta)\Tr\left(-\exp(u^{-1}(F^{E}+\mu_{v+uK}^{E})) -\exp(-u^{-1}(F^{E'}+\mu_{v+uK}^{E'})) \right).\ee

\begin{theorem} \label{evenChernMain} (i) The following equalities hold, 
\be \label{iota-Lie}\iota_v ch_{\nabla^\xi:\mathcal{G}}(\nabla^E, \nabla^{E'})=0, \ \ \ L_vch_{\nabla^\xi:\mathcal{G}}(\nabla^E, \nabla^{E'})=-ch_{\nabla^\xi:\mathcal{G}}(\nabla^E, \nabla^{E'}),\ee
 \be (d-\iota_v-u\iota_K+\Theta+u^{-1}\pi^*H)ch_{\nabla^\xi:\mathcal{G}}(\nabla^E, \nabla^{E'})=0.\ee
(ii) If $(\nabla_0^E, \nabla_0^{E'}), (\nabla_1^E, \nabla_1^{E'})$ are two horizontal gerbe module connections, then there exists $cs(\nabla_0^E, \nabla_0^{E'}; \nabla_1^E, \nabla_1^{E'})\in \widetilde{\Omega}^*(S\xi)[[u, u^{-1}]]$ such that 
\be  \label{trans} ch_{\nabla^\xi:\mathcal{G}}(\nabla_1^E, \nabla_1^{E'})-ch_{\nabla^\xi:\mathcal{G}}(\nabla_0^E, \nabla_0^{E'})=(d-\iota_v-u\iota_K+\Theta+u^{-1}\pi^*H)cs(\nabla_0^E, \nabla_0^{E'}; \nabla_1^E, \nabla_1^{E'}). \ee
\end{theorem}
\begin{proof}(i) Consider the local expression
$$ch_{\nabla^\xi:\mathcal{G}}(\nabla^E, \nabla^{E'})|_{\pi^{-1}(U_\alpha)}$$
$$=\exp(-u^{-1}\pi^*B_\alpha-2\pi i \theta_\alpha)\Tr\left(\exp(-u^{-1}(F^{E_\alpha}+\mu_{v+uK}^{E_\alpha})) -\exp(-u^{-1}(F^{E'_\alpha}+\mu_{v+uK}^{E'_\alpha})) \right).$$

Obviously, $\iota_v \pi^*B_\alpha=0.$
On the other hand, as $\nabla^{E_\alpha}$ is horizontal connection, we have $\nabla^{E_\alpha}_v=0$, but this equivalent to 
$$[\nabla^{E_\alpha}, \iota_v]=L_v.$$ Therefore
$$\iota_v(F^{E_\alpha})=[\iota_v, (\nabla^{E_\alpha})^2]=(L_v-\nabla^{E_\alpha}\iota_v)\nabla^{E_\alpha}-\nabla^{E_\alpha}(L_v-\iota_v\nabla^{E_\alpha})=[\nabla^{E_\alpha}, L_v]=0,$$
as $\nabla^{E_\alpha}$ is $\TT\times U(1)$-invariant. Similarly, $\iota_v(F^{E'_\alpha})=0$. We therefore have 
$$\iota_v ch_{\nabla^\xi:\mathcal{G}}(\nabla^E, \nabla^{E'})|_{\pi^{-1}(U_\alpha)}=0.$$ This shows the first equality in (\ref{iota-Lie}).
 
As $\nabla^{E_\alpha}$ is $\TT\times U(1)$-invariant, clearly  $L_v(F^{E_\alpha})=0$. The moment is
$$\mu_{v+uK}^{E_\alpha}=L_{v+uK}-[\iota_{v+uK}, \nabla^{E_\alpha}]. $$
Since $[v, K]=0$, it is easy to see that $$L_v\mu_{v+uK}^{E_\alpha}=0. $$ 
Now  $L_v \pi^*B_\alpha=0$ and $L_ve^{-2\pi i\theta_\alpha}=-e^{2\pi i\theta_\alpha},$ we have 
$$L_vch_{\nabla^\xi:\mathcal{G}}(\nabla^E, \nabla^{E'})|_{\pi^{-1}(U_\alpha)}=-ch_{\nabla^\xi:\mathcal{G}}(\nabla^E, \nabla^{E'})|_{\pi^{-1}(U_\alpha)}.$$ This shows the second equality in (\ref{iota-Lie}).

At last, as $(\xi, \nabla^\xi)$ and $(\{U_\alpha\}, H, B_\alpha, A_{\alpha\beta})$ are coupled on $M$, one has
$$2\pi i\theta_\alpha-\pi^*\iota_K B_\alpha=\Theta|_{\pi^{-1}(U_\alpha)},$$
where $\Theta$ is the connection 1-form on $S\xi$. Hence
\be \label{diff}
\begin{split}
&(d-\iota_v-u\iota_K)ch_{\nabla^\xi:\mathcal{G}}(\nabla^E, \nabla^{E'})|_{\pi^{-1}(U_\alpha)}=\\
=&(d-\iota_v-u\iota_K)\\
&\left[\exp(-u^{-1}\pi^*B_\alpha-2\pi i \theta_\alpha)\Tr\left(\exp(-u^{-1}(F^{E_\alpha}+\mu_{v+uK}^{E_\alpha})) -\exp(-u^{-1}(F^{E'_\alpha}+\mu_{v+uK}^{E'_\alpha})) \right)\right]\\
=&\left[\exp(-u^{-1}\pi^*B_\alpha-2\pi i \theta_\alpha)(-u^{-1}\pi^*dB_\alpha-2\pi id\theta_\alpha+\iota_K\pi^*B_\alpha)\right]\\
&\cdot \Tr\left(-\exp(u^{-1}(F^{E_\alpha}+\mu_{v+uK}^{E_\alpha})) -\exp(-u^{-1}(F^{E'_\alpha}+\mu_{v+uK}^{E'_\alpha})) \right)\\
=&\left[-u^{-1}\pi^*H-(2\pi i\theta_\alpha-\pi^*\iota_K B_\alpha))\right]\\
&\cdot \left[\exp(-u^{-1}\pi^*B_\alpha-2\pi i d\theta_\alpha)\Tr\left(-\exp(u^{-1}(F^{E_\alpha}+\mu_{v+uK}^{E_\alpha})) -\exp(-u^{-1}(F^{E'_\alpha}+\mu_{v+uK}^{E'_\alpha})) \right)\right]\\
=&\left(-u^{-1}\pi^*H-\Theta \right)ch_{\nabla^\xi:\mathcal{G}}(\nabla^E, \nabla^{E'})|_{\pi^{-1}(U_\alpha)},\\
\end{split}
\ee
and therefore 
$$(d-\iota_v-u\iota_K+\Theta+u^{-1}\pi^*H)ch_{\nabla^\xi:\mathcal{G}}(\nabla^E, \nabla^{E'})|_{\pi^{-1}(U_\alpha)}=0. $$

$\, $

(ii) Let
$$\nabla^E_t=(1-t)\nabla^E_0+t\nabla^E_1, \ \nabla^{E'}_t=(1-t)\nabla^{E'}_0+t\nabla^{E'}_1$$
and $F^E_t, F^{E'}_t, \mu^E_t, \mu^{E'}_t$ be the corresponding curvatures and momentums.

Let $$A^{E_\alpha}=\nabla_1^{E_\alpha}-\nabla_0^{E_\alpha},\   A^{{E'}_\alpha}=\nabla_1^{{E'}_\alpha}-\nabla_0^{{E'}_\alpha}. $$ 

We have
$$\phi_{\alpha\beta}^{-1}(-u^{-1}(F_t^{E_\alpha}+\mu_{v+uK, t}^{E_\alpha})-u^{-1}\pi^*B_\alpha-2\pi i \theta_\alpha) \phi_{\alpha\beta}=-u^{-1}(F_t^{E_\beta}+\mu_{v+uK, t}^{E_\beta})-u^{-1}\pi^*B_\beta-2\pi i \theta_\beta$$
and
$$\phi_{\alpha\beta}^{-1}(-u^{-1}A^{E_\alpha})\phi_{\alpha\beta} =-u^{-1}A^{E_\beta}.$$
Similar equalities hold for $E'$.

Therefore we have 
\be \label{trans}
\begin{split}
&\exp(-u^{-1}\pi^*B_\alpha-2\pi i \theta_\alpha)\\
&\cdot \int_0^1\Tr\left(-u^{-1}A^{E_\alpha}\exp(-u^{-1}(F_t^{E_\alpha}+\mu_{v+uK,t}^{E_\alpha}))+u^{-1}A^{E'_\alpha}\exp(-u^{-1}(F_t^{E'_\alpha}+\mu_{v+uK,t}^{E'_\alpha})) \right)dt
\end{split}
\ee
can be glued together as a global 
differential form in $\Omega^*(S\xi)[[u, u^{-1}]]$. Denote this form by $cs(\nabla_0^E, \nabla_0^{E'}; \nabla_1^E, \nabla_1^{E'})$. Since $\iota_v A^{E_\alpha}=0, L_v A^{E_\alpha}=0$, similar to proof of (i), we have 
$$\iota_vcs(\nabla_0^E, \nabla_0^{E'}; \nabla_1^E, \nabla_1^{E'})=0, \ L_vcs(\nabla_0^E, \nabla_0^{E'}; \nabla_1^E, \nabla_1^{E'})=-cs(\nabla_0^E, \nabla_0^{E'}; \nabla_1^E, \nabla_1^{E'})$$
and therefore 
$$cs(\nabla_0^E, \nabla_0^{E'}; \nabla_1^E, \nabla_1^{E'})\in \widetilde{\Omega}^*(S\xi)[[u, u^{-1}]].$$

Moreover, by the standard Chern-Simons transgression, we have
\be 
\begin{split}
&(d-\iota_v-u\iota_K)\\
&\int_0^1\Tr\left(-u^{-1}A^{E_\alpha}\exp(-u^{-1}(F_t^{E_\alpha}+\mu_{v+uK,t}^{E_\alpha}))+u^{-1}A^{E'_\alpha}\exp(-u^{-1}(F_t^{E'_\alpha}+\mu_{v+uK,t}^{E'_\alpha})) \right)dt\\
=&\Tr\left(\exp(-u^{-1}(F_1^{E_\alpha}+\mu_{v+uK,1}^{E_\alpha})) -\exp(-u^{-1}(F_1^{E'_\alpha}+\mu_{v+uK,1}^{E'_\alpha})) \right)\\
&-\Tr\left(\exp(-u^{-1}(F_0^{E_\alpha}+\mu_{v+uK,0}^{E_\alpha})) -\exp(-u^{-1}(F_0^{E'_\alpha}+\mu_{v+uK,0}^{E'_\alpha})) \right).\\
\end{split}
\ee
Then similar to (\ref{diff}), we see that
$$  (d-\iota_v-u\iota_K+\Theta+u^{-1}\pi^*H)cs(\nabla_0^E, \nabla_0^{E'}; \nabla_1^E, \nabla_1^{E'})=ch_{\nabla^\xi:\mathcal{G}}(\nabla_1^E, \nabla_1^{E'})-ch_{\nabla^\xi:\mathcal{G}}(\nabla_0^E, \nabla_0^{E'}).$$
\end{proof}

This theorem shows that $ch_{\nabla^\xi:\mathcal{G}}(\nabla^E, \nabla^{E'})$ is $(d-\iota_v-u\iota_K+\Theta+u^{-1}\pi^*H)$-closed in $\ometil^*(S\xi)^{\TT}[[u, u^{-1}]]$. Theorem \ref{fundtheorem} then tells us that $f^{-1}\left(ch_{\nabla^\xi:\mathcal{G}}(\nabla^E, \nabla^{E'})\right)$
 is $(\nabla^\xi-u\iota_K+u^{-1}H)$-closed in $\Omega^*(M, \xi)^{\TT}[[u, u^{-1}]]$.

We call  $$ CS(\nabla_0^E, \nabla_0^{E'}; \nabla_1^E, \nabla_1^{E'}):=f^{-1}\left(cs(\nabla_0^E, \nabla_0^{E'}; \nabla_1^E, \nabla_1^{E'}) \right)\in \Omega^*(M, \xi)[[u, u^{-1}]]$$
the {\bf exotic twisted equivariant Chern-Simons transgression term}. By (\ref{trans}) and Theorem \ref{fundtheorem} (formula (\ref{comm})), one has 
\be Ch_{\nabla^\xi:\mathcal{G}}(\nabla_1^E, \nabla_1^{E'})-Ch_{\nabla^\xi:\mathcal{G}}(\nabla_0^E, \nabla_0^{E'})=(\nabla^\xi-u\iota_K+u^{-1}H)CS(\nabla_0^E, \nabla_0^{E'}; \nabla_1^E, \nabla_1^{E'}). \ee

We define the
{\bf exotic twisted equivariant Chern character} to be:
$$Ch_{\nabla^\xi:\mathcal{G}}:K_{\TT}^0(M, \nabla^\xi:\mathcal{G})\to h_{\TT}^{*}(M, \nabla^\xi:H),$$
$$Ch_{\nabla^\xi:\mathcal{G}}(E, E')=\left[f^{-1}\left(ch_{\nabla^\xi:\mathcal{G}}(\nabla^E, \nabla^{E'})\right)\right].$$

\begin{remark}
A natural question is whether the exotic twisted equivariant Chern character is a rational isomorphism. 
However in the (untwisted) equivariant case, the equivariant Chern character 
$$
Ch^G: K^j_G(M) \to H_G^j(M) 
$$
where $H_G^j(M) $ is the even equivariant cohomology for $j=0$ and the odd equivariant cohomology for $j=1$, is {\em not} a rational isomorphism, as $H_G^j(M) = K^j_G(M) \otimes_{R(G)} R^\infty(G)$.
These results are due to Block \cite{Block} and Brylinski \cite{Bry87}. We do not explore this further in our context as it is not central to our investigations.
\bigskip
\end{remark}

\begin{remark} \label{commute1} Assume that the system $\{M, \mathcal{G}, (\xi, \nabla^\xi)\}$ is trivial, i.e., the $\TT$-action is trivial and the line bundle $\xi$ is trivial with trivial connection $\nabla^\xi=d$. Then a $(\TT\times U(1))$-equivariant gerbe module with horizontal connection for the gerbe $\{\hat L_{\alpha\beta}\}$ on $S\xi=M\times S^1$ in Definition \ref{maindef} can be identified with a gerbe module with connection for the gerbe $\mathcal{G}$ on $M$.  Therefore for the trivial system $\{M, \mathcal{G}, (\xi, \nabla^\xi)\}$, we have
\be K_{\TT}^0(M, \nabla^\xi:\mathcal{G})\cong K^0(M, \mathcal{G}).\ee\newline
\end{remark}

\begin{remark} \label{commute} Let us apply the constructions of exotic twisted equivariant K theory and exotic twisted equivariant Chern character to the concrete system $\{LZ, \bar H, (\cL^B, \nabla^{\cL^B})\}$ is Example 1. Let $i: Z\to LZ$ be the embedding.  One sees that when restricting to the fixed point submanifold $Z$ in $LZ$, the $\TT$-action as well as the honolomy line bundle $\cL^B$ become trivial and $i^*\bar H=H$. By Remark \ref{commute1}, 
\be K_{\TT}^0(Z, \nabla^{i^*\cL^B}: H)\cong K^0(Z, H). \ee
On the other hand, the triviality of the $\TT$-action and the line bundle on $Z$ imply that the moments $\mu_{v+uK}$'s in (\ref{glue1}) all disappear and all the $\theta_\alpha$'s are the same, denote it by $\theta$. Clearly $\exp(-2\pi i\theta)\otimes s=\gamma^{-1}$, where $s$ is the global identity section of the trivial circle bundle $Z\times S^1$. Then in view of (\ref{def-f}), we see that when restricted to $Z$, the exotic twisted equivariant Chern character degenerates to the usual twisted Chern character. This shows us that the diagram (\ref{twistedBC_H2}) is commutative. 
\end{remark}

\begin{remark} \label{noncommute} In Diagram (\ref{twistedBC_H2}), let $0\ne \Phi \in K_{\TT}^0(LZ, \nabla^{\cL^B}:\mathcal{G}) $ be in the kernel of $res$, that is $res(\Phi)=0 \in K^0(Z, H)$. 
Then $BCh_H(res(\Phi))=0 $, but $Ch_{\nabla^{\cL^B}:\mathcal{G}}(\Phi)\ne 0$.
\end{remark}

$\, $

\subsection{The odd case: gerbe modules}\label{odd k-thy}

Let $\cG= \{(H, B_\alpha, A_{\alpha\beta})\}$ be a gerbe with connection on $M$ as above.  Let $E=\{E_\alpha\}$ be a $U_{\gI}$ gerbe module with module connection $\nabla^E=\{\nabla^{E_\alpha}\}$.  Let $\phi=\{\phi_\alpha: E_\alpha\to E_\alpha\}$ be an automorphism of the gerbe module $E$ that respects the $U_{\gI}$ gerbe module structure, that is, $\phi_\alpha \in U_{\gI}(E_\alpha) = \{I + A_\alpha \in U(E_\alpha), A_\alpha\,\, \text{a trace class operator}\}$. We also need a compatibility condition on overlaps, 
$\psi_{\alpha \beta}\circ \phi_\alpha  \circ \psi_{\alpha \beta}^{-1} = \phi_\beta$. Here $\psi_{\alpha \beta}: L_{\alpha\beta}\otimes E_\beta \cong E_\alpha$, satisfying associativity by the gerbe condition. \bigskip

Then odd twisted K-theory $K^1(M, \cG)$ is the abelian group generated by such pairs $(E, \phi)$ with relations,
\begin{enumerate}
\item If $0\to E_1\to E_2 \to E_3\to 0$ is an exact sequence of gerbe modules such that the following diagram commutes, 
\begin{equation}
\begin{gathered}
\xymatrix{
0\ar[r] & E_1 \ar[r]\ar[d]^{\phi_1}\ar[r]& E_2\ar[r]\ar[d]^{\phi_2} & E_3 \ar[d]^{\phi_3}\ar[r] & 0\\
0\ar[r] &E_1 \ar[r]\ar[r]& E_2 \ar[r]  & E_3  \ar[r] & 0.
}
\end{gathered}
\end{equation}
 then one has $$(E_2, \phi_2) = (E_1, \phi_1) + (E_3, \phi_3).$$

\item $$(E, \phi_1\circ \phi_2) = (E, \phi_1) + (E, \phi_2).$$
\end{enumerate}
\bigskip

Then $\phi^{-1}\nabla^E\phi$ is another module connection for $E$. As explained in \cite{MS}, 
\be (\phi_\alpha^{-1}F^{E_\alpha}\phi_\alpha+B_\alpha )^k- (F^{E_\alpha}+B_\alpha )^k \ee
are differential forms with values in the trace class endomorphisms of $E_\alpha$ and 
\be \Tr[(\phi_\alpha^{-1}F^{E_\alpha}\phi_\alpha+B_\alpha )^k- (F^{E_\alpha}+B_\alpha )^k] \ee
patch together to be an even degree differential form on $M$. Denote it by $\Tr[(\phi^{-1}F^{E}\phi+B )^k- (F^{E}+B )^k].$

Let $\nabla^E(s)=s\phi^{-1}\nabla^E\phi+(1-s)\nabla^E$ be a path joining $\phi^{-1}\nabla^E\phi$ and $\nabla^E$. Let 
$$A(\phi)(s)=\partial_s\nabla^E(s)=\phi^{-1}\nabla^E\phi-\nabla^E,$$ which satisfies 
\be A(\phi)_\alpha(s)=\psi^{-1}_{\alpha\beta}A(\phi)_\beta(s) \psi_{\alpha\beta}. \ee

Following \cite{MS}, one defines the odd Chern character form
\be Ch_H(\nabla^E, \phi)=-\exp{(-B)}\int_0^1ds \Tr[A(\phi)\exp(-F^{E}(s))]. \ee

\subsection{The odd case: exotic twisted equivariant $K^1$-theory}

Let $M$ be a good $\TT$-manifold with an $\TT$-invariant cover $\{U_\alpha\}$. Let $\xi\to M$ be a $\TT$-equivariant Hermitian line bundle over $M$ equipped with a $\TT$-invariant Hermitian connection $\nabla^\xi$. Let $\mathcal{G}=(\{U_\alpha\}, H, B_\alpha, A_{\alpha\beta})$ be a weak $\TT$-invariant gerbe on $M$. Assume that $(\xi, \nabla^\xi)$ and $(\{U_\alpha\}, H, B_\alpha, A_{\alpha\beta})$ are coupled on $M$. 

Associated to the system $\{M, \mathcal{G}, (\xi, \nabla^\xi)\}$, we will introduce a version of twisted $K^1$-theory and odd twisted Chern character in this section. Adopt the same notations as in Section \ref{extevenK}.

\begin{definition} The pair $(E, \phi)$ with $E=\{E_{\alpha}, \nabla^{E_\alpha}\}$ a $U_{\gI}$ gerbe module  and $\phi=\{\phi_\alpha: E_\alpha\to E_\alpha\}$ a automorphism of the gerbe module respecting the $U_{\gI}$ structure, is said to be
a $(\TT\times U(1))$-equivariant {\bf odd gerbe module with horizontal connection} for the gerbe $\{\hat L_{\alpha\beta}\}$ if \newline
(a) the $(\TT\times U(1))$-invariant connections $\nabla^{E_\alpha}$'s vanish on the vertical direction, i.e. $\nabla^{E_\alpha}_v\equiv 0$;\newline
(b) there are  $(\TT\times U(1))$-equivariant  isomorphisms
$$ 
\psi_{\alpha\beta}: \hat L_{\alpha\beta} \otimes E_\beta \cong E_\alpha,
$$
which respect the connections. 
\end{definition}
Note that the isomorphisms $\{\psi_{\alpha\beta}\}$ are consistently defined on
triple overlaps because of the type  (\ref{gerbeproperty})  property of the gerbe  $\{(\hat L_{\alpha\beta}, \nabla^{\hat L_{\alpha\beta}}=d+\pi^*A_{\alpha\beta})\}$. \newline

One defines the {\bf exotic twisted $\TT$-equivariant $K^1$-theory} of $\{M, \mathcal{G}, (\xi, \nabla^\xi)\}, \mathcal{G})$, denoted as $K_{\TT}^1(M, \nabla^\xi:\mathcal{G})$, by
\be K_{\TT}^1(M, \nabla^\xi:\mathcal{G}):=\{(E, \phi)\}/\{\sim\}, \ee
where the equivalence relation $\sim$ is analogous to the description in \S\ref{odd k-thy}.

Similar to (\ref{trans}), the forms
\be \label{littleoddChern}
\exp(-u^{-1}\pi^*B_\alpha-2\pi i \theta_\alpha)\cdot \int_0^1ds\Tr\left(-u^{-1}A(\phi)_\alpha(s)\exp(-u^{-1}(F_s^{E_\alpha}+\mu_{v+uK,s}^{E_\alpha})) \right)
\ee
can be glued together as a global 
differential form in $\Omega^*(S\xi)[[u, u^{-1}]]$. Simply denote this form by 
\be ch_{\nabla^\xi:\mathcal{G}}(\nabla^E, \phi)=\exp(-u^{-1}\pi^*B-2\pi i \theta) \int_0^1ds\Tr\left(-u^{-1}A(\phi)(s)\exp(-u^{-1}(F_s^{E}+\mu_{v+uK,s}^{E})) \right)
.\ee
Then similar to the proof of Theorem \ref{evenChernMain}, one can prove that
\be \iota_v ch_{\nabla^\xi:\mathcal{G}}(\nabla^E, \phi)=0, \ \ \ L_vch_{\nabla^\xi:\mathcal{G}}(\nabla^E, \phi)=-ch_{\nabla^\xi:\mathcal{G}}(\nabla^E, \phi),\ee
and 
 \be (d-\iota_v-u\iota_K+\Theta+u^{-1}\pi^*H)ch_{\nabla^\xi:\mathcal{G}}(\nabla^E, \phi)=0.\ee
Therefore $ch_{\nabla^\xi:\mathcal{G}}(\nabla^E, \phi)$ is $(d-\iota_v-u\iota_K+\Theta+u^{-1}\pi^*H)$-closed in $\ometil^*(S\xi)^{\TT}[[u, u^{-1}]]$. Theorem \ref{fundtheorem} then tells us that $f^{-1}\left(ch_{\nabla^\xi:\mathcal{G}}(\nabla^E, \phi)\right)$
 is $(\nabla^\xi-u\iota_K+u^{-1}H)$-closed in $\Omega^*(M, \xi)^{\TT}[[u, u^{-1}]]$. Denote $$Ch_{\nabla^\xi:\mathcal{G}}(E, \phi)=\left[f^{-1}\left(ch_{\nabla^\xi:\mathcal{G}}(\nabla^E, \phi)\right)\right]\in h_{\TT}^{*}(M, \nabla^\xi:H).$$
 
Similar to Proposition 5.1 in \cite{MS}, one can show that $Ch_{\nabla^\xi:\mathcal{G}}(E, \phi)$ is independent of the choice of module horizontal connection $\nabla^E$ on $E$ and choice of automorphism $\phi$ of $E$. We define the {\bf exotic twisted equivariant odd Chern character} to be
$$Ch_{\nabla^\xi:\mathcal{G}}:K_{\TT}^1(M, \nabla^\xi:\mathcal{G})\to h_{\TT}^{*}(M, \nabla^\xi:H).$$


\end{document}